\documentclass{amsproc}
\usepackage{epsfig,amscd,amssymb,amsmath,amsfonts}
\usepackage{amssymb}
\usepackage{epsfig}
\newtheorem{theorem}{Theorem}[section]
\newtheorem{lemma}[theorem]{Lemma}

\theoremstyle{definition}

\newtheorem{corollary}[theorem]{Corollary}

\newtheorem{example}[theorem]{Example}

\theoremstyle{remark}
\newtheorem{remark}[theorem]{Remark}

\numberwithin{equation}{section}

\begin{document}

\title{Automorphisms of the $k$-algebra $k[X_1,...,X_m]$}

\author{Alina Petrescu-Nita}
     
\address{University Politehnica of Bucharest, Department of Mathematics, Splaiul Independentei 313, Bucharest, Romania}
\email{nita alina@yahoo.com}

\author{Mihai D. Staic}
\address{Department of Mathematics and Statistics, Bowling Green State University, Bowling Green, OH 43403 } \address{Institute of Mathematics of the Romanian Academy, PO.BOX 1-764, RO-70700 Bu\-cha\-rest, Romania.}

\email{mstaic@gmail.com, mstaic@bgsu.edu}




\subjclass[2010]{Primary  12D99}
\date{January 1, 1994 and, in revised form, June 22, 1994.}


\keywords{polynomial, Jacobian}

\begin{abstract} For a field $k$ of characteristic $0$, we present an algorithm for deciding if a morphism  $\phi:k[X_1,...,X_m]\to k[X_1,...,X_m]$  has an inverse. The algorithm also shows how to find the inverse when it exists. 
\end{abstract}

\maketitle


%

\section{Introduction}

A morphism of $k$-algebras $\phi:k[X]\to k[X]$ is an automorphism if and only if $\phi(X)=aX+b$ with $a\neq 0$, in this case the inverse is determined by $\phi^{-1}(X)=a^{-1}(X-b)$. Once we switch to several variables the problem is much more complicated. If $\phi:k[X_1,...,X_m]\to k[X_1,...,X_m]$ is an automorphism and $\phi(X_i)=F_i(X_1,...,X_m)$, then the determinant of the Jacobian matrix $Jac(F_1,...,F_m)$ is a nonzero element in $k$. The converse of this statement is known as the Jacobian conjecture. This problem is open and was intensively studied over the years, and it is infamous for several incorrect proofs proposed. For some partial results and reduction to simpler cases see \cite{b}, \cite{cd}, \cite{d}, \cite{w} and \cite{w1}.

The purpose of this paper is to present an algorithm for deciding if a $k$-algebra morphism $\phi:k[X_1,...,X_m]\to k[X_1,...,X_m]$ is invertible and to show how to  find the inverse.   We are not addressing the question whether the Jacobian conjecture is true or not. We only show how to decide if a particular $\phi$ is invertible, and if it is invertible,  how to find the inverse.

The paper is organized as follows. In section 2 we recall a few general results, most importantly Theorem \ref{grad} from \cite{b}. In section 3 we discuss  the case of two variables. We show that given a pair $F(X,Y)$, $G(X,Y)$, such that $Jac(F,G)$ is a nonzero element of $k$, then the system (\ref{eq1}) has a  unique solution in $k[X,Y][[t]]$. If $\phi$ is invertible then this solution is in $k[X,Y][t]$. After evaluating  at $t=1$ we get the inverse of $\phi$. In section 4 we state the results for the general case, and point the main differences from the case $m=2$.

\section{Preliminaries}

In this paper $k$ is field with $char(k)=0$. $k[X_1,...,X_m]$ is the $k$-algebra of polynomials over $k$, when $m=2$ we denote it with $k[X,Y]$. If $R$ is a ring then $R[[t]]$ is the ring of formal series with coefficients in $R$. 

Fix an integer  $m\geq 2$. Suppose that for every $1\leq i\leq m$ we have  $F_i\in k[X_1,...,X_m]$, then we can define a $k$-algebra morphism  $$\phi:k[X_1,...,X_m]\to k[X_1,...,X_m],$$ determined by $\phi(X_i)=F_i$ for all $1\leq i\leq m$. 

Recall that the Jacobian matrix associated to the $m$-tuple $(F_1,...,F_m)$ as the $m\times m$ matrix:
\begin{eqnarray*}
Jac(F_1,...,F_m)= \begin{pmatrix}
\frac{\partial{F_1}}{\partial{X_1}}(X_1,...,X_m)&...&\frac{\partial{F_1}}{\partial{X_m}}(X_1,...,X_m)\vspace{1mm}\\
\frac{\partial{F_2}}{\partial{X_1}}(X_1,...,X_m)&...&\frac{\partial{F_2}}{\partial{X_m}}(X_1,...,X_m)\vspace{1mm}\\
.&...&.\\
\frac{\partial{F_m}}{\partial{X_1}}(X_1,...,X_m)&...&\frac{\partial{F_m}}{\partial{X_m}}(X_1,...,X_m)\vspace{1mm}
 \end{pmatrix}.
\end{eqnarray*}
It is well know that if $\phi$ has an inverse then $det(Jac(F_1,...,F_m))\in k^*$. The converse of this statement is the well known Jacobian conjecture. 

If $F\in k[X_1,...,X_m]$ we denote by $deg(F)$ the maximum total degree of all mo\-no\-mi\-al that appears in $F$.  We denote by $deg(F_1,...,F_m)=max\{deg(F_1),...,deg(F_m)\}$. If $\phi:k[X_1,...,X_m]\to k[X_1,...,X_m]$, we denote by $deg(\phi)=deg(\phi(X_1),...,\phi(X_m))$. 
Recall from \cite{b} the following result.

\begin{theorem} (\cite{b}) If $\phi: k[X_1,...,X_m]\to k[X_1,...,X_m]$ is an automorphism then $deg(\phi^{-1})\leq deg(\phi)^{m-1}$. 
\label{grad}
\end{theorem} 

For more results about Jacobian conjecture see \cite{b}, \cite{cd}, \cite{d}, \cite{w} and \cite{w1}.

\section{Main Result}
In this section we study the case of two variables. First we make the following essential observation.  

\begin{lemma}  
Let $F(X,Y)$, $G(X,Y)\in k[X,Y]$, assume that there exist two polynomials $A(X,Y)$ and $B(X,Y)\in k[X,Y]$ such that 
\begin{equation}
 \left \{
  \begin{aligned}
   A(F(X,Y),G(X,Y))=X,\\ 
B(F(X,Y),G(X,Y))=Y, \\
F(A(X,Y),B(X,Y))=X,\\
G(A(X,Y),B(X,Y))=Y. 
  \end{aligned} \right.\label{eqinv}
\end{equation} 
Then there exist  $\mathcal{X}(t)$ and $\mathcal{Y}(t)\in k[X,Y][t]$ such that
\begin{eqnarray*}
&F(\mathcal{X}(t),\mathcal{Y}(t))=tX+(1-t)F(X,Y),&\\
&G(\mathcal{X}(t),\mathcal{Y}(t))=tY+(1-t)G(X,Y),&
\end{eqnarray*}
$\mathcal{X}(0)=X$, $\mathcal{Y}(0)=Y$, $\mathcal{X}(1)=A(X,Y)$ and $\mathcal{Y}(1)=B(X,Y)$. Moreover the $t$-degree for $\mathcal{X}(t)$ and $\mathcal{Y}(t)$ is less or equal to  the maximum of $deg(F(X,Y))$ and $deg(G(X,Y))$. 
\label{lemma1}
\end{lemma}

\begin{proof}
We define $\mathcal{X}(t)$ and $\mathcal{Y}(t)\in k[X,Y][t] (\subseteq k[X,Y][[t]])$ determined by:
\begin{eqnarray*}
&\mathcal{X}(t)=A(tX+(1-t)F(X,Y),tY+(1-t)G(X,Y)),&\\
&\mathcal{Y}(t)=B(tX+(1-t)F(X,Y),tY+(1-t)G(X,Y)).&
\end{eqnarray*}
First by (\ref{eqinv}) we have $F(\mathcal{X}(t),\mathcal{Y}(t))=F(A(tX+(1-t)F(X,Y),tY+(1-t)G(X,Y)), B(tX+(1-t)F(X,Y),tY+(1-t)G(X,Y)))=tX+(1-t)F(X,Y)$, and similarly $G(\mathcal{X}(t),\mathcal{Y}(t))=tY+(1-t)G(X,Y)$. 

Next we have that $\mathcal{X}(0)=A(0X+(1-0)F(X,Y),0Y+(1-0)G(X,Y))=A(F(X,Y),G(X,Y))=X$.
Similarly $\mathcal{Y}(0)=Y$, $\mathcal{X}(1)=A(X,Y)$ and $\mathcal{Y}(1)=B(X,Y)$. 

Finally, the $t$-degree for $\mathcal{X}(t)$ (and $\mathcal{Y}(t)$) is less or equal to the maximum of  $deg(A(X,Y)$ and  $deg(B(X,Y))$,  which by Theorem \ref{grad} is less or equal to the maximum of $deg(F(X,Y))$ and $deg(G(X,Y))$.
\end{proof}

This suggest that, given $F(X,Y)$ and $G(X,Y)\in k[X,Y]$, in order to find the inverse polynomial functions $A(X,Y)$ and $B(X,Y)$ we need to solve the system of equations
\begin{equation}
 \left \{
  \begin{aligned}
   F(\mathcal{U}(t),\mathcal{V}(t))=tX+(1-t)F(X,Y),\\
G(\mathcal{U}(t),\mathcal{V}(t))=tY+(1-t)G(X,Y),
  \end{aligned} \right.\label{eq1}
\end{equation} 
with the initial conditions
\begin{equation}
 \left \{
  \begin{aligned}
 \mathcal{U}(0)=X\\
\mathcal{V}(0)=Y
  \end{aligned} \right.\label{eq1c}
\end{equation} 
and then "evaluate" the solution at $t=1$. Unfortunately this is not quite true, but we have the following result. 
\begin{theorem} Let $F(X,Y)$ and $G(X,Y)\in k[X,Y]$  such that the determinant  of the jacobian matrix of the pair $(F(X,Y), G(X,Y))$ is in $k^*$. Then the system (\ref{eq1}) with initial condition 
(\ref{eq1c}), has a unique solution $\mathcal{U}(t)$, $\mathcal{V}(t)\in k[X,Y][[t]]$.  \label{main}
\end{theorem} 
\begin{proof} 
We are looking for 
$$\mathcal{U}(t)=\sum_{i\geq 0}u_i(X,Y)t^i, \; \; \mathcal{V}(t)=\sum_{i\geq 0}v_i(X,Y)t^i,$$ 
where $u_i(X,Y)$, $v_i(X,Y)\in k[X,Y]$. The plan is to show that for every $n\geq 1$, the system (\ref{eq1}) with initial condition (\ref{eq1c})  has a unique solution $mod\;t^n$. Moreover the solution $mod\;t^{n+1}$ is the extension of the solution  $mod\;t^n$.  

Because of the initial condition (\ref{eq1c}), we know that $u_0(X,Y)=X$ and $v_0(X,Y)=Y$ and so we have a unique solution $mod\;t$. 

Next we take the derivative of the equations (\ref{eq1}) to get: 

\begin{equation}
 \left \{
  \begin{aligned}
\frac{\partial{F}}{\partial{X}}(\mathcal{U}(t),\mathcal{V}(t))\mathcal{U}'(t)+\frac{\partial{F}}{\partial{Y}}(\mathcal{U}(t),\mathcal{V}(t))\mathcal{V}'(t)=X-F(X,Y),\vspace{2mm}\\
\frac{\partial{G}}{\partial{X}}(\mathcal{U}(t),\mathcal{V}(t))\mathcal{U}'(t)+\frac{\partial{G}}{\partial{Y}}(\mathcal{U}(t),\mathcal{V}(t))\mathcal{V}'(t)=Y-G(X,Y).
  \end{aligned} \right.\label{eq2}
\end{equation} 
We  can rewrite it as 
\begin{equation}
D_1(\mathcal{U}(t),\mathcal{V}(t))Z_{1,1}(X,Y)(t)=\begin{pmatrix}
  X-F(X,Y)\\
  Y-G(X,Y)
 \end{pmatrix}\label{eqm1}
\end{equation}
 where 
\begin{eqnarray*}
D_1(X,Y)= \begin{pmatrix}
\frac{\partial{F}}{\partial{X}}(X,Y)&\frac{\partial{F}}{\partial{Y}}(X,Y)\vspace{1mm}\\
\frac{\partial{G}}{\partial{X}}(X,Y)&\frac{\partial{G}}{\partial{Y}}(X,Y)
 \end{pmatrix},
\end{eqnarray*}
\begin{eqnarray*}
 Z_{1,1}(X,Y)(t)= \begin{pmatrix}
 \mathcal{U}'(t)\\
  \mathcal{V}'(t)
 \end{pmatrix}.
\end{eqnarray*}
Evaluating the equation (\ref{eqm1}) at $t=0$ we get:
\begin{eqnarray*}
D_1(X,Y)\begin{pmatrix}
 u_1(X,Y)\\
  v_1(X,Y)
 \end{pmatrix}=\begin{pmatrix}
  X-F(X,Y)\\
  Y-G(X,Y)
 \end{pmatrix}
\end{eqnarray*}
Since the Jacobian matrix has its determinant in $k^*$, we have that  $D_1(X,Y)$ has an inverse  in $M_2(k[X,Y])$. This implies that $Z_{1,1}(X,Y)(0)=\begin{pmatrix}
 u_1(X,Y)\\
  v_1(X,Y)
 \end{pmatrix}$ is uniquely determined, which proves that the system (\ref{eq1})  with initial condition (\ref{eq1c}) has a unique solution $mod\;t^2$. 

Next we take the second derivative of the system (\ref{eq1}) (or equivalently the derivative of (\ref{eqm1})) to get
\begin{equation}
D_1(\mathcal{U}(t), \mathcal{V}(t))Z_{1,2}(X,Y)(t)+D_2(\mathcal{U}(t), \mathcal{V}(t))Z_{2,2}(X,Y)(t)=0.\label{eqm2}
\end{equation}
Where 
\begin{eqnarray*}D_2(X,Y)= \begin{pmatrix}
\frac{\partial^2{F}}{\partial{X^2}}(X,Y)&\frac{\partial^2{F}}{\partial{X}\partial{Y}}(X,Y)&\frac{\partial^2{F}}{\partial{Y^2}}(X,Y) \vspace{1mm}\\
\frac{\partial^2{G}}{\partial{X^2}}(X,Y)&\frac{\partial^2{G}}{\partial{X}\partial{Y}}(X,Y)&\frac{\partial^2{G}}{\partial{Y^2}}(X,Y)
 \end{pmatrix},
\end{eqnarray*}

\begin{eqnarray*}Z_{1,2}(X,Y)(t)= \begin{pmatrix}
 \mathcal{U}''(t)\\
  \mathcal{V}''(t)
 \end{pmatrix},\; 
Z_{2,2}(X,Y)(t)= \begin{pmatrix}
(\mathcal{U}'(t))^2\\
  2\mathcal{U}'(t)\mathcal{V}'(t)\\
  (\mathcal{V}'(t))^2
 \end{pmatrix}.
\end{eqnarray*}
We evaluate the equation (\ref{eqm2}) at $t=0$ to get 
\begin{equation}
D_1(X,Y)Z_{1,2}(X,Y)(0)+D_2(X,Y)Z_{2,2}(X,Y)(0)=0.\label{eqm20}
\end{equation}
Notice that $Z_{2,2}(X,Y)(0)$ depends only on $u_1(X,Y)$ and $v_1(X,Y)$. Since $D_1(X,Y)$ is invertible we can solve uniquely equation (\ref{eqm20}) for $Z_{1,2}(X,Y)(0)= \begin{pmatrix}
 u_2(X,Y)\\
  v_2(X,Y)
 \end{pmatrix}$, which proves that the system (\ref{eq1})  with initial condition (\ref{eq1c}) has a unique solution $mod\;t^3$.

Next we do induction. Assume that $mod\;t^n$ we have a unique solution for the system (\ref{eq1})  with initial condition (\ref{eq1c}). Take the $n$-th derivative of the system (\ref{eq1}) to get 
\begin{equation}
D_1(\mathcal{U}(t), \mathcal{V}(t))Z_{1,n}(X,Y)(t)+\sum_{i=2}^nD_i(\mathcal{U}(t), \mathcal{V}(t))Z_{i,n}(X,Y)(t)=0.\label{eqmn}
\end{equation} 
Where for $1\leq i\leq n$
\begin{eqnarray*}D_i(X,Y)= \begin{pmatrix}
\frac{\partial^i{F}}{\partial{X^i}}(X,Y)&\frac{\partial^i{F}}{\partial{X^{i-1}}\partial{Y}}(X,Y)&...&\frac{\partial^i{F}}{\partial{Y^i}}(X,Y) \vspace{1mm}\\
\frac{\partial^i{G}}{\partial{X^i}}(X,Y)&\frac{\partial^i{G}}{\partial{X^{i-1}}\partial{Y}}(X,Y)&...&\frac{\partial^i{G}}{\partial{Y^i}}(X,Y)
 \end{pmatrix}\in M_{2\times (i+1)}(k[X,Y]),
\end{eqnarray*}
and 
\begin{eqnarray*} 
&Z_{i,n}(X,Y)(t)= 
\mathcal{U}'(t)\begin{pmatrix}
Z_{i-1,n-1}(X,Y)(t)\\
  0
 \end{pmatrix}+
\mathcal{V}'(t)\begin{pmatrix}
0\\
Z_{i-1,n-1}(X,Y)(t)
 \end{pmatrix}+&\\
 &\frac{d}{dt}Z_{i,n-1}(X,Y)(t)\in M_{(i+1)\times 1}(k[X,Y][[t]]).&
\end{eqnarray*}
For convenience we use the convention that $Z_{i,n}(X,Y)(t)=0$ for all $i\leq 0$ and $i\geq n+1$. 
Notice that 
\begin{eqnarray*} 
Z_{1,n}(X,Y)(t)= 
\begin{pmatrix}
\mathcal{U}^{(n)}(t)\\
\mathcal{V}^{(n)}(t)
\end{pmatrix},
\end{eqnarray*} in particular 
\begin{eqnarray*} 
Z_{1,n}(X,Y)(0)= 
\begin{pmatrix}
u_n(X,Y)\\
v_n(X,Y)
\end{pmatrix}.
\end{eqnarray*}
Moreover, for every $2\leq i\leq n$ the entries of the matrix $Z_{i,n}(X,Y)(t)$ can be expressed as polynomials in  $\mathcal{U}^{(j)}(t)$ and $\mathcal{V}^{(j)}(t)$ for $j\in \{1,...,n-i+1\}$. In particular, when we evaluate at $t=0$, the entries of the matrix $Z_{i,n}(X,Y)(0)$ are polynomials in $u_j(X,Y)$ and $v_j(X,Y)$ for $j\in \{1,...,n-i+1\}\subseteq  \{1,...,n-1\}$. 
\begin{equation}
D_1(X,Y)Z_{1,n}(X,Y)(0)+\sum_{i=2}^nD_i(X, Y)Z_{i,n}(X,Y)(0)=0.\label{eqmn1}
\end{equation} 
Since $D_1(X,Y)$ is invertible, and we already know $u_j(X,Y)$ and $v_j(X,Y)$ for $j\in   \{1,...,n-1\}$, we can solve uniquely the equation (\ref{eqmn1}) for $Z_{1,n}(X,Y)(0)=\begin{pmatrix}
 u_n(X,Y)\\
  v_n(X,Y)
 \end{pmatrix}$, and so the system (\ref{eq1})  with initial condition (\ref{eq1c}) has a unique solution $mod\;t^{n+1}$.
\end{proof}

This result, combined with Theorem \ref{grad}, gives an algorithm of deciding if a $k$-algebra morphism $\phi:k[X,Y]\to k[X,Y]$ is invertible or not. If it is invertible then we also  know how to find the inverse. More precisely we have the following.

\begin{corollary}
 Let $F(X,Y)$ and $G(X,Y)\in k[X,Y]$ of total degree at most $n$ such that $det(Jac(F(X,Y), G(X,Y)))\in k^*$. Let $\phi,\; \tau:k[X,Y]\to k[X,Y]$ be morphisms of $k$-algebras determined by 
$$\phi(X)=F(X,Y), \; \; \phi(Y)=G(X,Y),$$ 
$$\tau(X)=X+\sum_{i=1}^n u_i(X,Y), \; \; \tau(Y)=Y+\sum_{i=1}^n v_i(X,Y),$$ 
with $u_i(X,Y)$ and $v_i(X,Y)$ determined as in the proof of Theorem \ref{main}. Then the morphism $\phi$ has a polynomial inverse if and only if $\phi\circ \tau=id_{k[X,Y]}=\tau\circ\phi$. \label{cor1}
\end{corollary}
\begin{proof}
If $\phi$ has a polynomial inverse, then we know from  Lemma \ref{lemma1} that $(\mathcal{X}(t), \mathcal{Y}(t))$ is a solution for the system (\ref{eq1})  with initial condition (\ref{eq1c}). From Theorem \ref{main} the solution has to be unique, so $\mathcal{X}(t)=\mathcal{U}(t)$ and $\mathcal{Y}(t)=\mathcal{V}(t)$. 

Moreover from  Lemma \ref{lemma1} we know that the $t$-degree of $\mathcal{X}(t)$ and $\mathcal{Y}(t)$ is at most $n$ and so 
$\mathcal{X}(t)=X+\sum_{i=1}^n u_i(X,Y)t^i$ and $\mathcal{Y}(t)=Y+\sum_{i=1}^n v_i(X,Y)t^i$. In particular, using the notation from Lemma \ref{lemma1}, we have that $\phi^{-1}(X)=A(X,Y)=\mathcal{X}(1)=X+\sum_{i=1}^n u_i(X,Y)=\tau(X)$ and $\phi^{-1}(Y)=B(X,Y)=\mathcal{Y}(1)=Y+\sum_{i=1}^n v_i(X,Y)=\tau(Y)$. 

The converse is obvious. 
\end{proof}

The following result was first stated in \cite{b}, here we give a new proof. 
\begin{corollary} Let $k\subseteq L$ be an extension of fields of characteristic $0$, and $F(X,Y)$, $G(X,Y)\in k[X,Y]$. Assume that  there exist $A(X,Y)$, $B(X,Y)\in L[X,Y]$ such that the identities form (\ref{eqinv}) hold. Then $A(X,Y)$, $B(X,Y)\in k[X,Y]$. \label{kL}
\end{corollary}
\begin{proof}
We know from Theorem \ref{main} that the system of equations  (\ref{eq1}) with initial condition (\ref{eq1c}), has a unique solution $\mathcal{U}(t)$, $\mathcal{V}(t)$ in $k[X,Y][[t]]\subseteq L[X,Y][[t]]$. On the other hand,  Lemma \ref{lemma1} shows that the same system has a solution $\mathcal{X}(t)$, $\mathcal{Y}(t)$ in $L[X,Y][t]$. By Theorem  \ref{main}, the two solutions must be equal, so $\mathcal{U}(t)$, $\mathcal{V}(t)\in k[X,Y][t]$. In particular we get that $A(X,Y)=\mathcal{U}(1)\in k[X,Y]$ and $B(X,Y)=\mathcal{V}(1)\in k[X,Y]$.
\end{proof}
Next we give two examples for which we  compute $\mathcal{U}(t)$ and $\mathcal{V}(t)$. 
\begin{example} Suppose that $F(X,Y)=X+H(Y)$ and $G(X,Y)=Y$ where $H(Y)\in k[Y]$. Then $$\mathcal{U}(t)=X-tH(Y), \; \; \mathcal{V}(t)=Y.$$ In particular $A(X,Y)=X-H(Y)$ and $B(X,Y)=Y$. 
\end{example}

\begin{example} Suppose that $F(X,Y)=X+\frac{1}{a}(aX-bY)^n$ and $G(X,Y)=Y+\frac{1}{b}(aX-bY)^n$ where $a$, $b\in k^*$. Then 
$$\mathcal{U}(t)=X-\frac{t}{a}(aX-bY)^n, \; \; \mathcal{V}(t)=Y-\frac{t}{b}(aX-bY)^n.$$ In particular $A(X,Y)=X-\frac{1}{a}(aX-bY)^n$ and $B(X,Y)=Y-\frac{1}{b}(aX-bY)^n$. 

\end{example} 

\section{General Case}
The results from the previous section can be easily generalized to the case of polynomial in several variables. For completeness, we present without proof the  precise statements  and point out the main  differences. 

In this section we assume that $F_1, F_2, ..., F_m\in k[X_1,...,X_m]$ such that the maximum degree of the $m$-tuple $(F_1,F_2,...,F_m)$ is $n$. Just like above, we introduce the following system 
\begin{eqnarray}
&F_i(\mathcal{U}_1(t),...,\mathcal{U}_m(t))=tX_i+(1-t)F_i(X_1,...,X_m) \; \; {\rm for}\; {\rm all}\; 1\leq i\leq m,&
\label{eqm}
\end{eqnarray}
with initial conditions
\begin{eqnarray}
\mathcal{U}_i(0)=X_i \; \; {\rm for}\; {\rm all}\; 1\leq i\leq m.
\label{eqmc}
\end{eqnarray} 
\begin{theorem} Let $F_1, F_2, ..., F_m\in k[X_1,...,X_m]$  such that the determinant  of the jacobian matrix of the m-tuple $(F_1,F_2,...,F_m)$ is in $k^*$. Then the system (\ref{eqm}) with initial condition (\ref{eqmc}) has a unique solution 
$\mathcal{U}_1(t),\mathcal{U}_2(t),...,\mathcal{U}_m(t)$$\in k[X_1,X_2,...,X_m][[t]]$, with $\mathcal{U}_i(t)=X_i+\sum_{j\geq 1}u_{i,j}(X_1,...,X_m)t^j$.  \label{main2}
\end{theorem} 
\begin{proof}
The proof is identical with that for Theorem \ref{main}.
\end{proof}

\begin{corollary}
 Let Let $F_1, F_2, ..., F_m\in k[X_1,...,X_m]$ of total degree at most $n$ such that the determinant  of the jacobian matrix of the $m$-tuple $(F_1,F_2,...,F_m)$ is in $k^*$.
Let $\phi,\; \tau:k[X_1,...,X_m]\to k[X_1,...,X_m]$ be morphisms of $k$-algebras determined by $$\phi(X_i)=F_i(X_1,...,X_m),$$  $$\tau(X_i)=X_i+\sum_{j=1}^{n^{m-1}} u_{i,j}(X_1,...,X_m),$$  where $u_{i,j}(X_1,...,X_m)$ are as in Theorem \ref{main2}. Then the morphism $\phi$ has a polynomial inverse if and only if $\phi\circ \tau=id_{k[X_1,...,X_m]}=\tau\circ\phi$.
\end{corollary}
\begin{proof} Notice that in this case we have to take the first $n^{m-1}$ terms from the expression of $\mathcal{U}_i(t)$. When $m=2$ we recover Corollary \ref{cor1}.
\end{proof}

\begin{corollary} Let $k\subseteq L$ an extension of fields of characteristic $0$ and $F_1, F_2, ..., F_m\in k[X_1,...,X_m]$. Assume that  for all $1\leq i\leq m$ there  exist polynomials $A_i(X_1,...,X_m)\in L[X_1,...,X_m]$ such that 
$$F_i(A_1(X_1,....,X_m),..., A_m(X_1,....,X_m))=X_i$$ 
$$A_i(F_1(X_1,....,X_m),..., F_m(X_1,....,X_m))=X_i$$ 
for all $1\leq i\leq m$. Then $A_i(X_1,...,X_m)\in k[X_1,...,X_m]$ for all $1\leq i\leq m$. 
\end{corollary}

\begin{proof} The proof is similar with that for Corollary \ref{kL}.
\end{proof}
We end the paper with the following  straightforward generalization for Theorem \ref{main2}.
\begin{remark} Take $F_1, F_2, ..., F_m\in k[X_1,...,X_m]$ as in Theorem \ref{main2}, and  $P_1,$ $,...,$$P_m\in k[X_1,...,X_m]$. We define the following equation 
\begin{eqnarray}
F_i(\mathcal{W}_1(t),...,\mathcal{W}_m(t))=tX_i+(1-t)F_i(P_1,...,P_m)
\label{eqm3}
\end{eqnarray}
 for all $1\leq i\leq m$, with the initial conditions:
\begin{eqnarray}
\mathcal{W}_i(0)=P_i
\label{eqmc2}
\end{eqnarray}
for all $1\leq i\leq m$. Just like in Theorem (\ref{main2}), one can show the existence of a unique solution $\mathcal{W}_i(t)\in k[X_1,X_2,...,X_m][[t]]$ for the equation \ref{eqm3} with initial condition (\ref{eqmc2}).
\end{remark}




\bibliographystyle{amsalpha}

\end{document}